\documentclass[11pt]{amsart}

\usepackage[margin=2cm]{geometry}

\usepackage{palatino}
\usepackage{latexsym, amssymb, amsfonts, amsmath, amsthm, graphics, graphicx, subfigure, bbm, stmaryrd,tikz}
\usepackage{epstopdf}
\usepackage[pdftex,colorlinks,linkcolor=blue,menucolor=red,backref=false,bookmarks=true]{hyperref}

\newcommand{\field}[1]{\mathbb{#1}}
\newcommand{\R}{\field{R}}

\newcommand{\Var}{\mathrm{Var}}
\newcommand{\supp}{\mathrm{supp}}
\newtheorem{theorem}{Theorem}

\newtheorem{problem}{Open Problem}

\newtheorem{proposition}{Proposition}

\newtheorem{corollary}{Corollary}
\newtheorem{remark}{Remark}
\newtheorem{assumption}{Assumption}

\numberwithin{equation}{section}

\numberwithin{equation}{section}
\title{\bf Local  Functional Inequalities in One Dimensional Free Probability}
\author{Ionel Popescu}
\address{School of Mathematics \\
         Georgia Institute of Technology \\
         Atlanta, GA 30332-0160, USA \newline
 and \newline
Institute of Mathematics of Romanian Academy \\
21, Calea Grivitei Street\\
010702-Bucharest, Sector 1, Romania}         
\email{ipopescu@math.gatech.edu}

\thanks{ This work was supported by a grant of the Romanian National Authority for Scientific Research, CNCS Ð UEFISCDI, project number PN-II-RU-TE-2011-3-0259.  The author was also partially supported by Marie Curie Action grant nr. 249200 }

\date{}

\begin{document}

\maketitle

\begin{abstract}

In this note we  introduce and prove local and potential independent transportation,  Log-Sobolev and HWI inequalities in one dimensional free probability on compact intervals which are sharp.   We recover using this approach a free transportation inequality on the whole real line which was put forward recently by M. Maida and E. Maurel-Segala in \cite{MS}.  

Our method is bases on the operator theoretic approach developed in \cite{LP2} to deal with the free Poincar\'e's inequality.  

\end{abstract}

\section{Introduction}

There is a large literature on functional inequalities in the classical case.   Some of these inequalities apply to the case of  random matrices and produce, as the dimension grows to infinity,  interesting functional inequalities in the limit.   The main connection between the random matrices and  free probability is  due to the main result of Voiculescu \cite{Voi1, Voi2} which states that large random matrices become asymptotically free as the dimension grows to infinity.  This is a very rich bridge from free probability to random matrices and back.  

Among the classical counterparts of classical inequalities, we mention the transportation which was first discussed in \cite{BV} for the quadratic potentials and was inspired by the work of Otto and Villani \cite{OV}.    We  describe now the statement of this inequality as it is an important point in the economy of this note.      

For a given potential $V:\R\to\R$, the logarithmic energy with external field $V$ of a probability measure on $\R$ is defined by 
\[
E_{V}(\mu)=\int Vd\mu-\iint \log|x-y|\mu(dx)\mu(dy).  
\]
It is a standard result (cf. \cite{ST}) that under some mild growth of $V$, there is a unique probability measure $\mu_{V}$, which minimizes the functional $E_{V}$.   If we set $E_{V}(\mu|\mu_{V})=E_{V}(\mu)-E_{V}(\mu_{V})$ this plays the analog of the entropy in the classical case.  The transportation inequality associated to $V$ states that there is a positive $\rho$ such that for any compactly supported probability measure $\mu$ 
\begin{equation}\label{e0:1}
\rho W_{2}^{2}(\mu,\mu_{V})\le E_{V}(\mu|\mu_{V})   
\end{equation}
where $W_{2}(\mu,\nu)$ is the Wasserstein distance based on quadratic cost function given by 
\[
W_{2}(\mu,\nu)=\left\{\inf_{\pi\in\Pi(\mu,\nu)}\int |x-y|^{2}\pi(dx\,dy)\right\}^{1/2}
\]
for measures $\mu,\nu$ of finite second moment.   This was first proved for the case of $V(x)=x^{2}/2$ by Biane and Voiculescu in \cite{BV} using complex Burger's equation  and then for the case of $V(x)-\rho x^{2}$ convex in \cite{HPU1} using random matrices.  Yet, another direct approach is using tools from mass transportation tools and is given in \cite{popescu, LP}.    

Another classical inequality which found a natural analog in the free probability world is the Log-Sobolev which was introduced in a certain form by Voiculescu in \cite{Voi5} and then proved to be equivalent to the one which is most common now by Bianne and Speicher in \cite{B3}.   With the notation from above, it states that there is a positive $\rho$ such that for any probability measure $\mu$, 
\begin{equation}\label{e0:2}
4\rho E_{V}(\mu|\mu_{V})\le I(\mu|\mu_{V})
\end{equation}
where 
\[
I(\mu|\mu_{V})=\int (H\mu-V')^{2}d\mu \text{ with } H\mu(x)=\begin{cases} p.v.\int\frac{2}{x-y}\mu(dy) & \frac{d\mu}{dx}\in L^{3}(\R)  \\ +\infty & \text { otherwise }\end{cases}
\]
where the integral in the definition of the Hilbert transform $H\mu$ is in the principal value sense.     This inequality has received a random matrix proof in \cite{B2} for the case of the case of $V(x)-\rho x^{2}$ is convex and then using tools from the mass transportation in \cite{LP}.   

Notice that so far these inequalities require some convexity on the potential $V$.   A natural question is to ask if there is a transportation or Log-Sobolev without the convexity assumption on $V$.  For the transportation case there is a version put forward recently by M. Ma\"ida and E. Maurel-Segala in \cite{MS} in which the main condition on $V$ is a quadratic growth at infinity and the base metric $W_{2}$ is replaced by $W_{1}$, a weaker metric. They use this to say something about the concentration of the empirical distribution of eigenvalues of random matrices with general potentials.  

For the Log-Sobolev case, without convexity assumption, the statement from \eqref{e0:2} can not be true as it is.  What is the natural replacement of \eqref{e0:2} if we drop the convexity on $V$ is not clear. 

At one end is the case of (strong) convex potentials where both \eqref{e0:1} and \eqref{e0:2} are well understood but as soon as we loose the convexity property, the inequalities in discussion become problematic.   At the other end of the spectrum is the case of these inequalities which are in fact potential independent.   It is this topic which is under investigation here.   

Now we describe a little bit the main results and how the paper is organized.    

To formulate the question clearly, the first change is that in place of the entropy $E_{V}(\mu|\mu_{V})$ we use a very closely related quantity which for any two probability measures $\mu,\nu$, is given by
\[
\mathcal{H}(\mu,\nu)=-\iint \log|x-y|(\mu-\nu)(dx)(\mu-\nu)(dy). 
\]
As one can see, this is independent of the potential $V$, but it is not really very different from $E_{V}(\mu|\mu_{V})$ (see for instance, \eqref{e2:8} below).    

One of the main results of this note is that for probability measures supported on $[-2,2]$, 
\begin{equation}\label{e0:3}
W_{1}^{2}(\mu,\nu)\le 2\mathcal{H}(\mu,\nu)
\end{equation}
where the inequality is actually sharp.  This inequality can be seen as a local version of the transportation inequality which is at the same time potential independent.  By scaling,  this can be extended to probability measures on any compact interval, and thus it can be really interpreted as some form of universal transportation inequality.     This is treated in Section~\ref{s:1}.

In Section~\ref{s:2} we show that the metric $W_{1}$ is optimal in \eqref{e0:3} and can not be replaced by any other $W_{p}$ with $p>1$.   

The interesting fact is that now if we take a potential $V$, with at least quadratic growth at infinity, then we can actually turn the local version of \eqref{e0:3} into a global transportation inequality which states that for some $C>0$ and any probability measure  $\mu$ on $\R$,  
\[
C W_{1}^{2}(\mu,\mu_{V})\le E_{V}(\mu|\mu_{V}).   
\]
which is the result from \cite{MS}.   This is the content of Section~\ref{s:3}.   We should point out that the approach of M. Maida and E. Maurel-Segal from \cite{MS} to prove this result uses some random matrices where here we do not appeal to any of that.   The idea we use here is borrowed from the mass transportation techniques to combine the local transportation with the growth of $V$ at infinity.       

Section~\ref{s:4} is dedicated to a local version of the Log-Sobolev.   The first thing we need to set properly is the analog of the Fisher information, $I_{V}(\mu|\mu_{V})$.   When restricted to $[-2,2]$, the version we propose is the following 
\[
\mathcal{J}(\mu,\nu)=\begin{cases}
\int (H\mu-H\nu)^{2}d\alpha &\text{ if } H\mu,H\nu\in L^{2}(\alpha)  \\
+\infty &\text{otherwise}
\end{cases}
\]
where $\alpha$ is the semicircle law on $[-2,2]$ and $H\mu$ is the Hilbert transform of the measure $\mu$.  One of the main results in this section is that 
\[
\mathcal{J}(\mu,\nu)=\begin{cases} 
2\int \left( \frac{d\mu}{d\beta}-\frac{d\nu}{d\beta} \right)^{2}d\beta, & \text{ if }\quad \frac{d\mu}{d\beta},\frac{d\nu}{d\beta}\in L^{2}(\beta) \\
+\infty &\text{otherwise} 
\end{cases} 
\]
with $\beta$ being the arcsine law on $[-2,2]$.  This last equality is nothing but the interesting property of the Hilbert transform which says that the map 
\[
H_{\beta}:L^{2}(\beta)\to L^{2}(\alpha)\quad H_{\beta}\phi=H(\phi\,d\beta)
\]
is an isometry up to a multiplication by a constant.   Using these properties, we prove the following local version of the free Log-Sobolev.  For any probability measures on $[-2,2]$, 
\begin{equation}\label{e0:5}
2\mathcal{H}(\mu,\nu)\le \mathcal{J}(\mu,\nu).
\end{equation}
As in the case of local transportation this turns out to be sharp.   

We continue this discussion in Section~\ref{s:6} of the local Log-Sobolev in which the $L^{2}$ norm of $H\mu-H\nu$ from the definition of $\mathcal{J}$ above is  replaced by the square of the $L^{p}$ norm with $1<p<2$.   It is shown that if such a Log-Sobolev holds true, then necessarily $3/2\le p$ but it is posted as an open problem if $p=3/2$ is the optimal threshold for which the inequality is satisfied.   At any rate, even though $p=3/2$ does not produce an $L^{p}$ version of the Log-Sobolev, it is still natural to look for the smallest which does produce such an inequality.   

Finally in Section~\ref{s:5} we discuss a version of the celebrated Otto-Villani HWI inequality which links together $W_{1}$, $\mathcal{H}$ and $\mathcal{J}$.  This is a refinement of the Log-Sobolev inequality.   

On the technical side, the main tools we use here are borrowed from the operator theoretical approach to the free Poincar\'e inequality put forward in \cite{LP} and, as we already mentioned, for the global version of transportation inequality we employ some tools from the classical mass transportation.

\section{Potential Independent Transportation Inequality on $[-L,L]$}\label{s:1}

We will treat here essentially the case of measures on $[-2,2]$.  The case of measures on $[-L,L]$ following by simple scaling.  

Given a $p\ge1$ and two measures $\mu,\nu$ on the real line such that $\int |x|^{p}\,\mu(dx)$ and $\int |x|^{p}\,\nu(dx)$ are both finite, we define 
\begin{equation}\label{e1:W}
W_{p}(\mu,\nu)=\left\{\inf_{\pi\in\Pi(\mu,\nu)}\int |x-y|^{p}\pi(dx\,dy)\right\}^{1/p}
\end{equation}
where here $\Pi(\mu,\nu)$ denotes the set of probability measures on $\R^{2}$ with marginals $\mu,\nu$.   $W_{p}$ is a metric for the weak topology on the set of probability measures with $p$th finite moment.  

We will be interested in $W_{1}$ which can also be characterized as
\begin{equation}\label{e1:1}
W_{1}(\mu,\nu)=\sup\left\{ \int g\,d(\mu-\nu),|g(x)-g(y)|\le |x-y|\right\}. 
\end{equation}
$W_{1}$ is a distance for the topology of weak convergence of probability measures with finite first moment. 

Another description of the distance $W_{p}$ is given by the following (see for example \cite[page 75]{Vi2}).  If $\mu$, $\nu$ are two probability measures on $\R$ such that $\nu$ does not have atoms, then there is a unique non-decreasing map $\theta$ such that $\theta_{\#}\nu=\mu$ (i.e. $\mu(A)=\nu(\theta^{-1}(A))$).  In addition,  
\begin{equation}\label{e1:0}
W_{1}(\mu,\nu)=\int|\theta(x)-x|\nu(dx).
\end{equation}

Next we define the \emph{free reduced relative entropy} of two compactly supported measures $\mu,\nu\in\mathcal{P}(\R)$ to be given by 
\begin{equation}\label{e1:2}
\mathcal{H}(\mu,\nu)=-\iint \log|x-y|(\mu-\nu)(dx)(\mu-\nu)(dy).  
\end{equation}
It is well known \cite[Lemma 1.8]{ST} or \cite[Lemma 6.41]{Deift1} that $\mathcal{H}(\mu,\nu)\ge0$ with equality if and only if $\mu=\nu$.    For the integrability properties see a detailed discussion in \cite[page 142]{Deift1} the only thing we point out here being that $\mathcal{H}(\mu,\nu)$ is finite if and only if
\[
\iint |\log|x-y||\mu(dx)\mu(dy)<\infty\text{ and }\iint |\log|x-y||\nu(dx)\nu(dy)<\infty.  
\]
If either of these conditions fail, we set $\mathcal{H}(\mu,\nu)=+\infty$.  

For measures on the interval $[-2,2]$, the reduced relative entropy can be well understood in terms of the operator structure associated to the logarithmic potential of measures on $[-2,2]$.  To this matter we recall here some of the main results discussed in \cite{LP2} which were put forward in order to deal with the free Poincar\'e inequality.  

We will work with the following reference measures on $[-2,2]$
\[
\alpha(dx)=\frac{\sqrt{4-x^{2}}}{2\pi}dx, \quad \text{ and }\quad \beta(dx)=\frac{dx}{\pi\sqrt{4-x^{2}}}.  
\]
Most of the action takes place around the arcsine measure $\beta$ and we will use $\langle, \rangle$ to denote the inner product in $L^{2}(\beta)$.   In the sequel we will use the following notation 
\begin{equation}\label{e1:30}
\phi_{n}(x)=T_{n}\left(\frac{x}{2} \right) \quad
    \text{ and } \quad \psi_{n}(x)=U_{n}\left(\frac{x}{2} \right) \quad \text{ for }n\ge0.
\end{equation}
where 
$T_n(x)$, the {\em Chebychev polynomials of the first kind}, are defined by $T_n(\cos\theta)=\cos(n\theta)$ and $U_{n}$, the \emph{Chebyshev polynomials of second kind}, are described by  $U_{n}(\cos \theta)=\frac{\sin (n+1)\theta}{\sin\theta}$.    Adjusting a little bit the polynomials $T_{n}$ as $\tilde{T}_{0}=T_{0}$ and $\tilde{T}_{n}(x)=\sqrt{2}T_{n}(x)$, then it is easy to see that $\{\tilde{T}_{n}(x/2)\}$ form an the orthonormal basis for $L^{2}(\beta)$.   Similarly,  $U_{n}(x/2)$ form an orthonormal basis for $L^{2}(\alpha)$.   Another relation which plays an important role here is  
\begin{equation}\label{e1:5}
\phi_{n}'=\frac{n}{2}\psi_{n-1}.  
\end{equation}

First, we introduce the operators $\mathcal{E,N,L}$ on $C^{2}$ functions on $[-2,2]$ as follows.  Given a $C^{2}$ function $\phi:[-2,2]\to\R$, set
\begin{equation}\label{e:EN}
\begin{split}
(\mathcal{E}\phi)(x)&=-\int \log|x-y|\phi(y)\beta(dy) , \\
(\mathcal{N}\phi)(x)&=\int y\phi'(y)\beta(dy) 
+x\int \phi'(y)\beta(dy)
- (4-x^{2})\int \frac{\phi'(x)-\phi'(y)}{x-y} \,  \beta(dy) \\ 
(\mathcal{L}\phi)(x)&=-(4-x^{2})\phi''(x)+x\phi'(x).
\end{split}
\end{equation}

For convenience in what follows we will use the space 
\[
K=\{ f\in L^{2}(\beta): \int f\,d\beta=0 \}
\]
which is the orthogonal to constants in $L^{2}(\beta)$.   The reason we single out this space is that the operators $\mathcal{N}$ and $\mathcal{E}$ (properly extended) are the inverse of each other.   

Now we summarize the main properties presented in \cite[Proposition 1]{LP2} and needed in this note.  

\begin{proposition}\label{p:1}
\begin{enumerate}
\item $\mathcal{E}$ sends $C^{2}([-2,2])$ into $C^{2}([-2,2])$ and can be extended to a bounded selfadjoint  operator from $L^{2}(\beta)$ into itself.   
\item For any $C^{2}$ function $\phi \in K$, 
\begin{equation}\label{ep:6}
\begin{split}
\mathcal{E}\mathcal{N}\phi &= \phi, \\
\mathcal{N}\mathcal{E}\phi &=\phi.
\end{split}
\end{equation}
\item In addition $\mathcal{E}\phi_{0}=0$, while for $n\ge1$, $\mathcal{E}\phi_{n}=\frac{1}{n}\phi_{n}$ and   $\mathcal{N}\phi_{n}=n\phi_{n}$ for any $n\ge0$.  In other words, $\mathcal{N}$ is the counting number operator for the Chebyshev basis in $L^{2}(\beta)$.
\item $\mathcal{N}$ can be canonically extended to a selfadjoint operator on $L^{2}(\beta)$ which restricted to $K$ has inverse $\mathcal{E}$.  
\item $\mathcal{L}=\mathcal{N}^{2}$ and it satisfies for any $C^{2}$ functions $\phi,\psi\in L^{2}(\beta)$, 
\begin{equation}\label{ep:8}
\langle \mathcal{L}\phi,\psi\rangle=2\int \phi' \psi'd\alpha.
\end{equation}
\end{enumerate}
\end{proposition}

In what follows, a key role is played but he following Corollary. 
\begin{corollary}\label{c:1}
For any $C^{2}$ functions $\phi,\psi\in K$, 
\begin{equation}\label{e1:11}
\langle \phi,\psi \rangle=2\int  (\mathcal{E}^{2}\phi)'\psi'\,d\alpha=2\int (\mathcal{E}\phi)'(\mathcal{E}\psi)'\,d\alpha.
\end{equation}
\end{corollary}

\begin{proof}
This follows from \eqref{ep:8} with $\phi$ replaced by $\mathcal{E}^{2}\phi$ and the fact that $\mathcal{E}$ and $\mathcal{N}$ are the inverse of each other.  
\end{proof}

The first main result of this note is the following.

\begin{theorem}\label{t:1}
For any two probability measures $\mu,\nu$ on $[-2,2]$, 
\begin{equation}\label{e1:13}
W_{1}^{2}(\mu,\nu)\le 2\mathcal{H}(\mu,\nu).  
\end{equation}
Equality is attained for any two probability measures $\mu,\nu$ such that $\mu(dx)-\nu(dx)=cx\beta(dx)$ for some constant $c$.
\end{theorem}

\begin{proof}  
We should notice first that the maximization from \eqref{e1:1} can be taken over the set of smooth functions $g$ with the property that $\int gd\beta=0$.  Hence, 
\begin{equation}\label{e1:8}
W_{1}(\mu,\nu)=\sup\left\{ \int g\,d(\mu-\nu),g\in K \text{ and } |g'|\le 1\right\}. 
\end{equation}

Now we prove first \eqref{e1:13} for the case of measures $\mu$, $\nu$ which have smooth densities with respect to the reference measure $\beta$.   Therefore,  write $\mu-\nu=\psi d\beta$ and continue with 
\[
W_{1}(\mu,\nu)=\sup\left\{\int g\psi \,d\beta,|g'|\le1 \right\},
\]
where now the supremum is taken over all smooth functions $g$.   

Furthermore, for any smooth  function $g$ with bounded derivative by $1$, using \eqref{e1:11} we have the following string of equalities and inequalities
\begin{equation}\label{e1:3}
\begin{split}
\int g\psi \,d\beta&=2\int g' \left(\mathcal{E}^{2}\psi\right)'d\alpha \\
&\le 2\int \left|\left(\mathcal{E}^{2}\psi\right)'\right|d\alpha\le 2\left(\int \left(\left(\mathcal{E}^{2}\psi\right)'\right)^{2}d\alpha\right)^{1/2} =\sqrt{2}\langle \mathcal{E}\psi,\mathcal{E}\psi\rangle^{1/2} \\ 
&=\sqrt{2}\langle\mathcal{E}^{2}\psi,\psi\rangle^{1/2}. 
\end{split}
\end{equation}

The inequality we want to prove now certainly follows from the fact that 
\begin{equation}\label{e1:4}
\langle\mathcal{E}^{2}\psi,\psi\rangle \le \langle \mathcal{E}\psi,\psi \rangle,
\end{equation}
which is a consequence of the fact that the spectrum of $\mathcal{E}:K\to K$ is $\{1,1/2,1/3,\dots\}$ with eigenfunctions  $\{\phi_{1},\phi_{2},\phi_{3},\dots\}$.   More precisely, if we write $\psi =\sum_{n\ge1}\gamma_{n}\phi_{n}$, then \eqref{e1:4}  is equivalent to  
\[
\sum_{n\ge1}\frac{1}{n^{2}}\gamma_{n}^{2}\le \sum_{n\ge1}\frac{1}{n}\gamma_{n}^{2}.  
\]
which is obvious.  In fact this inequality is saturated when $\gamma_{n}$ are all $0$ with the exception of $n=1$.  Thus tracing back all the inequalities in between we get equality in \eqref{e1:13} for any measures $\mu,\nu$ for which $\mu-\nu=cxd\beta$ with some constant $c$.

In the second place, an approximation procedure shows that we can reduce the proof to the case of measures having smooth densities with respect to $\beta$.  

To carry out this reduction, notice that if $\mathcal{H}(\mu,\nu)$ is infinite then there is nothing to prove here, so we will assume that $\mathcal{H}(\mu,\nu)$ is finite in which case, according to \cite[Equation 6.47]{Deift1}, 
\begin{equation}\label{e1:14}
\mathcal{H}(\mu,\nu)=\int_{0}^{\infty}\frac{|\hat{\mu}(t)-\hat{\nu}(t)|^{2}}{t}dt.
\end{equation}

In order to use this equation we consider a smooth compactly supported function $\zeta:[0,1]\to [0,\infty]$ such that  $\int_{0}^{1}\zeta(x)dx=1$.  Now,  set $\zeta_{\epsilon}(x)=\frac{1}{\epsilon}\zeta(x/\epsilon)$ for small $\epsilon$ and consider the measure $\xi_{\epsilon}(dx)=\zeta_{\epsilon}(x)dx$.  Based on the measure $\xi_{\epsilon}$ we construct $\mu_{\epsilon}=\xi_{\epsilon} \star \mu$ and similarly $\nu_{\epsilon}=\xi_{\epsilon}\star \nu$.  It is now clear in the first place that $\mu_{\epsilon}$ and $\nu_{\epsilon}$ are probability measures with smooth compact support  such that $\mu_{\epsilon}\xrightarrow[\epsilon\to0]{}\mu$ and $\nu_{\epsilon}\xrightarrow[\epsilon\to0]{}\nu$ in the weak topology.  In particular this means that 
\begin{equation}\label{e1:16}
W_{1}(\mu_{\epsilon},\nu_{\epsilon})\xrightarrow[\epsilon\to0]{}W_{1}(\mu,\nu).
\end{equation}
At the same time 
\begin{equation}\label{e1:15}
\mathcal{H}(\mu_{\epsilon},\nu_{\epsilon})\xrightarrow[\epsilon\to0]{}\mathcal{H}(\mu,\nu)
\end{equation}
which follows from \eqref{e1:14} and 
\[
\mathcal{H}(\mu_{\epsilon},\nu_{\epsilon})=\int_{0}^{\infty}\frac{|\hat{\mu}(t)-\hat{\nu}(t)|^{2}|\hat{\zeta}_{\epsilon}|^{2}}{t}dt \xrightarrow[\epsilon\to0]{} \int_{0}^{\infty}\frac{|\hat{\mu}(t)-\hat{\nu}(t)|^{2}}{t}dt
\]
where in the last part we used $|\hat{\zeta}_{\epsilon}|\le 1$ and $\lim_{\epsilon\to0}\hat{\zeta}_{\epsilon}=1$ combined with the dominated converge.  

Hence,  we have smooth compactly supported approximations of $\mu,\nu$, the only problem is that the approximations $\mu_{\epsilon}$, $\nu_{\epsilon}$ have support in $[-2-\epsilon, 2+\epsilon]$.  Thus what we can do is to take $\ell_{\epsilon}(x)=x/(1+\epsilon)$ which maps $[-2-\epsilon,2+\epsilon]$ into $[-2+\epsilon/(1+\epsilon),2-\epsilon/(1+\epsilon)]$ and therefore the push forward $\tilde{\mu}_{\epsilon}=(\ell_{\epsilon})_{\#}\mu_{\epsilon}$ and $\tilde{\nu}_{\epsilon}=(\ell_{\epsilon})_{\#}$ are probability measures with smooth densities supported on $[-2+\epsilon/(1+\epsilon),2-\epsilon/(1+\epsilon)]$.   On the other hand it is easy to see that 
\[
W_{1}(\mu_{\epsilon},\nu_{\epsilon})=(1+\epsilon)W_{1}(\tilde{\mu}_{\epsilon},\tilde{\nu}_{\epsilon})
\]
and 
\[
\mathcal{H}(\tilde{\mu}_{\epsilon},\tilde{\nu}_{\epsilon})=\mathcal{H}(\mu_{\epsilon},\nu_{\epsilon})+\log(1+\epsilon).
\]

The conclusion is that now we can apply the smooth case to the measures $\tilde{\mu}_{\epsilon}$ and $\tilde{\nu}_{\epsilon}$ to deduce that 
\[
W_{1}^{2}(\mu_{\epsilon},\nu_{\epsilon})\le 2(1+\epsilon)^{2}(\mathcal{H}(\mu_{\epsilon},\nu_{\epsilon})+\log(1+\epsilon))
\]
which combined with \eqref{e1:16} and \eqref{e1:15} leads to \eqref{e1:13} for any probability measures $\mu,\nu$ on $[-2,2]$.\qedhere
\end{proof}

A corollary of the proof of Theorem~\ref{t:1} is the following nice representation of $W_{1}$ in terms of the operator $\mathcal{E}$ which we want to record and use later as a separate result.  

\begin{corollary}\label{c:2}  If $\mu,\nu$ are two probability measures on $[-2,2]$ such that $d\mu-d\nu=\psi d\beta$, where $\psi$ is in $L^{2}(\beta)$, then 
\begin{equation}\label{e1:w}
W_{1}^{2}(\mu,\nu)=2\langle \mathcal{E}^{2}\psi,\psi \rangle.  
\end{equation}
\end{corollary}

\begin{proof}
In the case of smooth $\psi$, the proof is nothing but the content of the key sequence of inequalities from \eqref{e1:3}.  

For the general case, we need to approximate $\psi\in L^{2}(\beta)$ by smooth functions.  This requires a little care but it is straightforward and we point only the main steps.   

First, notice that by simple approximations, it is sufficient to prove the statement for measures $\mu,\nu$ which are supported inside $(-2,2)$.  

Second,  take the mollifier $\zeta_{\epsilon}$ as in the proof of Theorem~\ref{t:1} and consider the standard mollification of $\psi$ as $\psi_{\epsilon}= \zeta_{\epsilon}\star \psi$.  It is clear that $\psi_{\epsilon}$ converges a.s. to $\psi$ and in $L^{2}(\beta)$.  To see the last part, the convergence in $L^{2}(\beta)$, one needs to observe that due to Cauchy's inequality,  $ |\psi_{\epsilon}(x)|^{2} \le 2\pi \|\psi \|_{L^{2}(\beta)}^{2}$, from which $\psi_{\epsilon}$ is certainly in $L^{2}(\beta)$ and its norm is controlled  by the norm of $\psi$.  This is sufficient to conclude that $\psi_{\epsilon}$ converges in $L^{2}(\beta)$ toward $\psi$.   \qedhere
\end{proof}

By simple scaling we have the following consequence of Theorem~\ref{t:1}. 

\begin{corollary} If $\mu,\nu$ are two probability measures supported on $[-L,L]$, then 
\begin{equation}\label{e1:17}
W_{1}^{2}(\mu,\nu)\le 2L^{2}\mathcal{H}(\mu,\nu).  
\end{equation}
The constant $2L^{2}$ in front of $\mathcal{H}$ is sharp.  
\end{corollary}

\begin{remark}\label{r:10}  In the classical case, the  Csisz\'ar-Kullback-Pinsker inequality states that 
\[
\|\mu-\nu \|_{v}^{2}\le 2\mathcal{H}(\mu|\nu)
\]
for any two measures on the real line where $\mathcal{H}(\mu|\nu)$ is the classical relative entropy given by $\mathcal{H}(\mu|\nu)=\int\frac{d\mu}{d\nu}\log \frac{d\mu}{d\nu}d\nu$ in the case there $\mu$ is absolutely continuous with respect to $\nu$ and is $+\infty$ otherwise.   Also the distance $\|\mu-\nu \|_{v}$ is the total variation distance.   

\' Edouard Maurel-Segala and Myl\`ene Ma\"ida asked if there is such an inequality in the free case.   Given the setup in this section the answer turns out to be negative.   

To see this choose some probability measures $\mu,\nu$ on $[-2,2]$ such that $\mu-\nu=c\phi_{n}\,d\beta$ with $n\ge1$.   With this choice, 
\[
\|\mu-\nu \|_{v}=|c|\int |\phi_{n}|\,d\beta \text{ while }\mathcal{H}(\mu,\nu)=c^{2}\langle \mathcal{E}\phi_{n},\phi_{n} \rangle=\frac{c^{2}}{2n}.  
\]
Since it is not hard to check that for any $n\ge1$
\[
\int |\phi_{n}|\,d\beta=\frac{1}{\pi}\int_{0}^{\pi}|\cos(n t)|dt=\frac{2}{\pi},
\]
the conclusion is that  there is no constant $C>0$ such that 
\[
C\|\mu-\nu \|_{v}^{2}\le \mathcal{H}(\mu|\nu)
\]
for all probability measures $\mu,\nu$ on $[-2,2]$.   

In the classical case Csisz\'ar-Kullback-Pinsker on compact intervals implies the transportation with $W_{1}$ metric and this is also the reason why the latter is eclipsed by the former.   In the free context, since the Csisz\'ar-Kullback-Pinsker fails, it makes the $W_{1}$ transportation more interesting.   
\end{remark}

\section{No local transportation with respect to the $W_{p}$-metric, $p>1$ }\label{s:2}

Given the above local transportation from Theorem~\ref{t:1}, one natural question in this framework is whether one can extend it to the case of $W_{p}$ metric instead of $W_{1}$ with $p>1$.

In other words, is there a constant $C>0$ such that for any measures $\mu,\nu$ supported in $[-2,2]$, the following holds true
\begin{equation}\label{e1b:1}
C W_{p}^{2}(\mu,\nu)\le \mathcal{H}(\mu,\nu)?
\end{equation}

As we will see, the answer is no.  To see why this is the case, we start with the following \cite[page 75]{Vi2}
\[
W_{p}^{p}(\mu,\nu)=\int_{0}^{1} |F_{\mu}^{-1}(t)-F_{\nu}^{-1}(t)|^{p}dt,
\]
where $F_{\mu},F_{\nu}$ are the cumulative functions of $\mu,\nu$ and $F_{\mu}^{-1}$, $F_{\nu}^{-1}$ are their generalized inverses.   Now, take $d\mu=\phi\, d\beta$, $\phi>0$ on $[-2,2]$ with $\phi\in C^{\infty}([-2,2])$.  Then choose $\nu_{\epsilon}=(\phi+\epsilon h)d\beta$, where $h$ is a smooth compactly supported function on $(-2,2)$ with $\int h\,d\beta=0$ and $\epsilon$ is small enough.   With this choice it is obvious that 
\begin{equation}\label{e2b:2}
\mathcal{H}(\mu,\nu_{\epsilon})=\epsilon^{2}\langle \mathcal{E}h,h\rangle.
\end{equation}

On the other hand, it is not hard to show that the cumulative functions of $\mu$, $\nu$ are smooth and they actually depend smoothly also on $\epsilon$.  Denoting now for simplicity $F(x)=F_{\mu}(x)$, $F_{\epsilon}(x)=F_{\nu_{\epsilon}}(x)$, and $G(x)=\int_{-2}^{x}h(y)\beta(dy)$, then 
\[
F_{\epsilon}(x)=F(x)+\epsilon G(x).
\]
From this, it is not hard to see that 
\[
F_{\epsilon}^{-1}(t)=F^{-1}(t)-\epsilon \frac{G(F^{-1}(t))}{f(F^{-1}(t))}+O(\epsilon^{2})
\]
where $f(x)=F'(x)$.  Consequently, it follows that 
\[
\begin{split}
W_{p}^{2}(\mu,\nu_{\epsilon})&=\epsilon^{2}\int_{0}^{1}\frac{|G(F^{-1}(t))|^{p}}{|f(F^{-1}(t))|^{p}}dt+O(\epsilon^{4})=\int\frac{G^{2}(x)}{f(x)}dx\\ 
&=\frac{\epsilon^{2}}{\pi}\int_{-2}^{2}\frac{\left|\int_{-2}^{x}h(y)\beta(dy)\right|^{p}}{\phi^{p-1}(x)}(4-x^{2})^{(p-1)/2}dx+O(\epsilon^{4})
\end{split}
\]

If we assume \eqref{e1b:1} is true for some constant $C>0$, then letting $\epsilon$ go to $0$ and the above considerations lead to the following conclusion.  For any $\phi>0$ such that $\int \phi \,d\beta=1$ and any compactly supported smooth function $h$ on $(-2,2)$, we get that
\[
\frac{C}{\pi}\int_{-2}^{2}\frac{\left|\int_{-2}^{x}h(y)\beta(dy)\right|^{p}}{\phi^{p-1}(x)} (4-x^{2})^{(p-1)/2}dx\le \langle \mathcal{E}h,h \rangle.  
\] 
Then the conclusion is that in fact this got to be true for any smooth function $h$ on $[-2,2]$ and any smooth positive $\phi$ on $[-2,2]$ such that $\int \phi \,d\beta=1$.   However, if we fix a function $h$, say $h(x)=x$, then the above inequality implies that there is a positive constant $C'>0$ such that
\[
\int_{-2}^{2}\frac{(4-x^{2})^{(2p-1)/2}}{\phi^{p-1}(x)}dx\le C'
\]
for any smooth positive function $\phi$ such that $\int \phi\,d\beta=1$.  But we can choose a function $\phi$ which is arbitrary close to $0$ on the interval $[-1/3,1/3]$ and the above inequality leads to a contradiction.  This argument shows that \eqref{e1b:1} can not be true, so there is no local transportation with $W_{p}$-metric, $p>1$.  

From a certain perspective, this argument is just a linearization of the transportation inequality and this in the classical case corresponds to a Poincar\'e inequality for the reference measure $\mu$.  In a similar fashion, the argument outlined above says that not every measure $\mu$ on $[-2,2]$ satisfies a free Poincar\'e inequality.   The right version of such a Poincar\'e inequality in this framework is actually the one discussed in details in \cite{LP2}, but we do not enter into further discussion here.

\section{A global version of Transportation inequality on the whole real line}\label{s:3}

Now we want to prove a transportation inequality on the whole real line which in this case is potential dependent.  To setup the scene, we will take a potential $V:\R\to\R$ such that 
\begin{assumption}\label{a:1} $V$ is bounded below, measurable and satisfies
\[
\liminf_{|x|\to\infty}\frac{V(x)}{x^{2}}>0.  
\]
\end{assumption}

For any probability measure $\mu\in\mathcal{P}(\R)$, 
\[
E_{V}(\mu)=\int Vd\mu-\iint \log|x-y|\mu(dx)\mu(dy).  
\]
It is known, see for example \cite[Theorem 1.3]{ST} that there is a unique measure (also with compact support) $\mu_{V}$ which minimizes $E_{V}(\mu)$ over all probabilities $\mu\in\mathcal{P}(\R)$.

The equilibrium measure $\mu_{V}$  (cf. 
\cite[Thm.I.1.3]{ST}) satisfies  
\begin{equation}\label{e2:1}
\begin{split}
V(x)&\ge 2\int \log|x-y|\mu_{V}(dy)+K_{V} \quad\text{quasi-everywhere on }\R \\ 
V(x)&= 2\int \log|x-y|\mu_{V}(dy)+K_{V} 
\quad\text{quasi-everywhere on}\:\:\supp{\mu}_{V}, \\ 
\end{split}
\end{equation} 
where $K_{V}$ is known as Robin constant.   

For the definition of the notion of quasi-everywhere see for instance \cite[page 25]{ST}.  In what follows, to simplify  a little bit the exposition, we will denote 
\[
U(x)=2\int \log|x-y|\mu_{V}(dy)+K_{V}.  
\]

For simplicity in what follows, we denote $E_{V}(\mu_{V})=E_{V}$ and we set \emph{the relative free entropy} 
\[
E_{V}(\mu|\mu_{V})=E_{V}(\mu)-E_{V}(\mu_{V}).  
\]

For any measure $\mu$, using the above equalities we can write now, 
\begin{equation}\label{e2:8}
E_{V}(\mu|\mu_{V})=\int (V(x)-U(x))\mu(dx)+\mathcal{H}(\mu,\mu_{V}). 
\end{equation}
Because of Assumption~\ref{a:1} and \eqref{e2:1},  there are  constants $A,B>0$ such that 
\begin{equation}\label{e2:13}
V(x)-U(x)\ge A\mathbbm{1}_{[-B,B]^{c}}(x)x^{2}.
\end{equation}
For a given measure $\mu\in\mathcal{P}(\R)$ with compact support, say in $[-L,L]$, one thing we can try to do is to use \eqref{e1:17} and \eqref{e2:1} to estimate from below $\mathcal{H}(\mu,\mu_{V})$.  Hence we get at first that 
\[
E_{V}(\mu|\mu_{V})\ge\int (V-U)d\mu+\frac{1}{2L^{2}}W_{1}^{2}(\mu,\mu_{V}).  
\]
The problem here is that as $L$ goes to infinity the $W_{1}$ term simply disappears.  On the other hand, for large values of $L$, the potential $V$ is bounded from below by a quadratic, while $U$ grows at most logarithmically, thus for large values of $L$, $V-U$ grows at least quadratically.  With this in mind, the idea is to refine the above scheme so that it takes better advantage of the quadratic growth. 

To carry this idea through, we do the following.  Take a large $L\ge B$ such that $\mu_{V}$ is supported on $[-L/2,L/2]$.  From \eqref{e2:13}, we learn that $V(x)\ge Ax^{2}$ for $|x|\ge L$.    Consider the function $\phi:[-\sqrt{3}L,\sqrt{3}L]\to\R$ (see Figure \ref{f1}) given by 
\[
\phi(x)=\begin{cases}
-\frac{2L^{3}}{3L^{2}-x^{2}},& -\sqrt{3}L<x\le -L \\
x,& -L<x\le L \\
\frac{2L^{3}}{3L^{2}-x^{2}},& L<x\le \sqrt{3}L. \\
\end{cases}
\]
\begin{figure}[h!]\caption{The graph of $\phi$.}\label{f1} 
  \begin{tikzpicture}[domain=-3:3]
    \draw[->] (-2.5,0) -- (2.5,0) node[right] {$x$}; 
    \draw[->] (0,-3.2) -- (0,3.2) node[above] {$y$};
    \draw[black, thick, domain=-1:1]    plot (\x,\x); 
    \draw[black,thick, domain=1:1.55,samples=200] plot (\x, {2/(3-\x*\x)}) node[right] {$y=\phi(x)$};
    \draw[black, thick, domain=-1.55:-1,samples=200] plot (\x, {-2/(3-\x*\x)});
    \draw[dashed] (1,1)--(1,0) node[below=1pt] {$L$};
    \draw[dashed] (-1,-1)--(-1,0) node[left=7pt,below=1pt] {$-L$};
    \draw[dashed] (1.57,3)--(1.57,0) node[right=5pt,below=-1pt] {$L\sqrt{3}$};
    \draw[dashed] (-1.57,-3)--(-1.57,0) node[left=20pt, below=-1pt] {$-L\sqrt{3}$};
  \end{tikzpicture}
  \end{figure}

There are several elementary properties of this function we will use in the sequel.  For $L\ge3^{1/4}$,
\begin{equation}\label{e2:9}
\begin{split}
1)&\quad\phi \text{ is odd increasing and $C^{1}$ on } (-\sqrt{3}L,\sqrt{3}L)\\ 
2)&\quad \phi'(x)\text{ is decreasing on } (-\sqrt{3}L,0]\text{ and is increasing on } [0,\sqrt{3}L)   \\
3)& \quad\lim_{x\to-\sqrt{3}L}\phi(x)=-\infty, \lim_{x\to \sqrt{3}L}\phi(x)=\infty \\
4)&\quad\phi(x)=x \text{ for }-L\le x\le L \\
5)& \quad \log\left|\frac{\phi(x)-\phi(y)}{x-y}\right|\le 2\log|\phi(x)|+2\log|\phi(y)|\text{ if }|x|\vee |y|\ge L. 
\end{split}
\end{equation}

Next, define $\nu=(\phi^{-1})_{\#}\mu$ and notice that $\nu$ is supported on $[-\sqrt{3}L,\sqrt{3}L]$.   Now we  start using \eqref{e2:1} to  justify that 
\begin{equation}\label{e2:10}
E_{V}(\mu|\mu_{V})=\int (V(\phi(x))-U(x))\nu(dx)-\iint \log\left|\frac{\phi(x)-\phi(y)}{x-y} \right|\nu(dx)\nu(dy) + \mathcal{H}(\nu,\mu_{V}).
\end{equation}
The point of this is that $\nu$ is supported on $[-\sqrt{3}L,\sqrt{3}L]$ and from \eqref{e1:17}, 
\begin{equation}\label{e2:11}
\mathcal{H}(\nu,\mu_{V}) \ge \frac{1}{6L^{2}}W_{1}^{2}(\nu,\mu_{V}).
\end{equation}

Now let's turn our attention to the first two terms of \eqref{e2:10} and notice that because $\phi(x)=x$ for $x\in[-L,L]$ and $V(x)\ge U(x)$, combined with the last line of \eqref{e2:9}, yields 
\begin{align*}
\int (V(\phi(x))&-U(x))\nu(dx)-\iint \log\left|\frac{\phi(x)-\phi(y)}{x-y} \right|\nu(dx)\nu(dy) \\ & \ge \int_{|x|\ge L}(A|\phi(x)|^{2}-2\log (2\sqrt{3} L)-C)\nu(dx)-4\int_{|x|\ge L}\log|\phi(x)|\nu(dx) \ge \frac{A}{2}\int_{|x|\ge L}\phi^{2}(x)\nu(dx)
\end{align*}
for large $L$ and a certain constant $C>0$.    Putting together these findings, we conclude that for large $L$,  
\[
E_{V}(\mu|\nu)\ge \frac{A}{2}\int_{|x|\ge L}\phi^{2}(x)\nu(dx)+ \frac{1}{6L^{2}}W_{1}^{2}(\nu,\mu_{V}).
\]

At this point we use the characterization of the Wasserstein distance given by \eqref{e1:0}.   Namely, if $\theta$ is the transportation map of $\mu_{V}$ into $\nu$, then $\phi\circ\theta$ is the transportation map of $\mu_{V}$ into $\mu$.  Thus, the right hand side of the above equation can be continued as 
\begin{align*}
 \frac{A}{2}\int_{|x|\ge L}\phi^{2}(x)\nu(dx) &+ \frac{1}{6L^{2}}\left(\int |\theta(x)-x|\mu_{V}(dx)\right)^{2} \\
 &=  \frac{A}{2}\int_{|\theta(x)|\ge L}\phi^{2}(\theta(x))\mu_{V}(dx)+ \frac{1}{6L^{2}}\left(\int_{|\theta(x)|\le L} |\theta(x)-x|\mu_{V}(dx)\right)^{2} \\ 
 &= \frac{A}{2}\int_{|\theta(x)|\ge L}\phi^{2}(\theta(x))\mu_{V}(dx)+ \frac{1}{6L^{2}}\left(\int_{|\theta(x)|\le L} |\phi(\theta(x))-x|\mu_{V}(dx)\right)^{2} \\ 
 &\ge \frac{2A}{9}\int_{|\theta(x)|\ge L}(\phi(\theta(x))-x)^{2}\mu_{V}(dx)+ \frac{1}{6L^{2}}\left(\int_{|\theta(x)|\le L} |\phi(\theta(x))-x|\mu_{V}(dx)\right)^{2} \\
& \ge \frac{2A}{9}\left(\int_{|\theta(x)|\ge L}|\phi(\theta(x))-x|\mu_{V}(dx)\right)^{2}+ \frac{1}{6L^{2}}\left(\int_{|\theta(x)|\le L} |\phi(\theta(x))-x|\mu_{V}(dx)\right)^{2} \\
&\ge \frac{2A}{12AL^{2}+9} W_{1}^{2}(\mu,\mu_{V}).
\end{align*}
To clarify this long equation we make the following points.  The second equality in the the above follows from the fourth property of \eqref{e2:9}. The third inequality follows from the fact that $3|\phi(y)|\ge2|\phi(y)-x|$ for any $|x|\le L/2$ and $|y|\ge L$ which (because $\phi$ is increasing and odd) it is easy to see it is equivalent to $3\phi(L)\ge 2(\phi+L/2)$ which is obviously clear from $\phi(L)=L$.   The second inequality is just Cauchy's inequality and the last one is easy to see from $am^{2}+bn^{2}\ge(m+n)^{2}/(1/a+1/b)$ for $a,b,m,n\ge0$.   

What we just proved is the following global version of the transportation inequality.

\begin{theorem}\label{t:2}
Under Assumption~\ref{a:1}, there is a constant $C>0$ depending on $V$ such that   for any probability measure $\mu$ on $\R$, 
\begin{equation}
C W_{1}^{2}(\mu,\mu_{V})\le E_{V}(\mu|\mu_{V}).
\end{equation}
\end{theorem}

\begin{remark}  What is worth pointing here is that the constant $C$ depends on the choice of $L$, which in turn is determined by the constants $A,B$ from \eqref{e2:13}, the constant $K_{V}$ from \eqref{e2:1} and the support of $\mu_{V}$.   

\end{remark}

\section{ Local versions of free Log-Soblev inequality}\label{s:4}

The main feature of the transportation inequality in Theorem ~\ref{t:1} is that it's formulation does not depend on any potential on the interval $[-2,2]$.   Perhaps better said, the quantities $\mathcal{H}(\mu,\nu)$ and $W_{1}(\mu,\nu)$ are defined independently of the potential $V$ defining the relative entropy in \eqref{e2:1}.   However the potential independent result of Theorem~\ref{t:1} combined with a growth at infinity, provides the necessary ingredients for a transportation inequality with a potential $V$ as is presented in Theorem~\ref{t:2}. 

On the other hand in \cite{B3} it is introduced the free Log-Sobolev is introduced and deduced from an inequality of \cite{Voi5}.  The main statement  is the following.   We say that the free Log-Sobolev associated to  a potential $V:\R\to\R$ holds true if there is a constant $\rho>0$ such that for any probability measure $\mu\in\mathcal{P}(\R)$, 
\begin{equation}\label{e3:1}
\rho E_{V}(\mu|\mu_{V})\le I(\mu|\mu_{V})
\end{equation}
where 
\begin{equation}\label{e3:2}
I(\mu|\mu_{V})=\begin{cases}
\int (H\mu-V')^{2}d\mu   &\text{ if } \frac{d\mu}{dx}\in L^{3}(\R) \\ 
+\infty & \text{otherwise} \end{cases}.
\end{equation}
and 
\begin{equation}\label{e3:HT}
H\mu(x)= p.v. \int\frac{2}{x-y}\mu(dy)
\end{equation}
with the integral taken in the principal value sense.   

This inequality was then reproved by Biane in \cite{B2} for the case of strongly convex potentials $V$ using random matrix approximations and also in \cite{LP} using  tools form mass transportation.    

One immediate consequence of the above inequality and the uniqueness of the equilibrium measure  is that if $V'=H\mu$ holds almost surely on the support of $\mu$, then $\mu$ must be equal to $\mu_{V}$.  

It was pointed out in \cite{B2} by Biane that in the case of nonconvex potentials, $LSI$ can not hold true due essentially to the fact that one can construct different measures whose Hilbert transforms agree to $V'$ on their supports and this violates \eqref{e3:1} and the uniqueness of the minimizer of $E_{V}(\mu)$.  The example given there is the one in which $V$ is a double well potential with $V(x)=(x-a_{1})^{2}/2$ on $J_{1}=[a_{1}-2,a_{1}+2]$ and $V(x)=(x-a_{2})^{2}/2$ on $J_{2}=[a_{2}-2,a_{2}+2]$ with $|a_{1}-a_{2}|>4$.   With this choice it is clear that $V'$ is equal to the Hilbert transform on $J_{1}$ of the semicircular measure on $J_{1}$.   Similarly, $V'$ is equal to the Hilbert transform on $J_{2}$ of the semicircular measure on $J_{2}$.  

Therefore Log-Sobolev inequality can not be true without extra assumptions on the potential $V$.   In order to get a version which is potential independent, one reasonable thought is to try to replace the derivative of $V'$ by something which is in terms of the measure $\mu_{V}$ and a natural candidate to that is the Hilbert transform of $\mu_{V}$.  Furthermore the roles played by $\mu$ and $\mu_{V}$ should be symmetric and thus the integral with respect to $\mu$ in \eqref{e3:2} should be replaced by an integral with respect to a measure which is symmetric with respect to both, $\mu$ and $\mu_{V}$.   One  way of fixing this is to integrate with respect to a fixed measure rather than a measure depending on $\mu$ and $\mu_{V}$.   

To summarize the preceding paragraph,  we introduce a potential independent Log-Sobolev in the form of 
\[
\rho \mathcal{H}(\mu,\nu)\le \int \left( H\mu-H\nu \right)^{2}d\zeta
\]
where here we take a reference measure $\zeta$ supported on a certain set $K$ and the measures $\mu,\nu$ to be supported by $K$.   

We do not investigate here the general situation alluded above but focus on the following version of the free information for measures supported on $[-2,2]$ with the reference measure $\zeta$ being the semicircular measure $\alpha$:
\begin{equation}\label{e3:13}
\mathcal{J}(\mu,\nu)=\begin{cases}
\int (H\mu-H\nu)^{2}d\alpha &\text{ if } H\mu,H\nu\in L^{2}(\alpha)  \\
+\infty &\text{otherwise}.
\end{cases}
\end{equation}

If $\alpha_{L}$ is the semicircular law on $[-2L,2L]$, then the extension of the above is 
\begin{equation}\label{e3:13b}
\mathcal{J}_{L}(\mu,\nu)=\begin{cases}
\int (H\mu-H\nu)^{2}d\alpha_{L} &\text{ if } H\mu,H\nu\in L^{2}(\alpha_{L})  \\
+\infty &\text{otherwise}.
\end{cases}
\end{equation}

Assume for the moment that $d\mu=\phi\,d\beta$ with $\phi$ a $C^{2}$ function on $[-2,2]$.  In this case, since $\mathcal{E}\phi$ is a $C^{2}$ function,  the principle value integral defining $H\mu$ for $x\in(-2,2)$ can be shown to be equal to
\begin{equation}\label{e3:11}
\frac{1}{2}(H\mu)(x)=\frac{d}{dx}(\mathcal{E}\phi)(x)=(\mathcal{U}\phi)(x):=\int \frac{\phi(x)-\phi(y)}{x-y}\beta(dy).
\end{equation}
Given a function $\phi$, we will use the notation  $H_{\beta}\phi$ for  $H(\phi\beta)$.   

The operator $\mathcal{U}$ appears in \cite{LP2} and the main property it satisfies is that 
\begin{equation}\label{e3:10}
\| \mathcal{U}f\|_{\alpha}^{2}=\frac{1}{2}\Var_{\beta}( f )
\end{equation}
for any $f\in L^{2}(\beta)$ where here $\Var_{\beta}(f)$ is the variance of $f$ with respect to $\beta$.     
In particular, what this says is that the operator $\mathcal{U}$ is an isometry between the spaces orthogonal to constants in $L^{2}(\beta)$ and $L^{2}(\alpha)$.   

To see this property, it suffices to check that
\begin{equation}\label{e3:100}
\mathcal{U}\phi_{n}=\frac{1}{2}\psi_{n-1}
\end{equation}
which actually follows from $\frac{d}{dx}(\mathcal{E}\phi)(x)=(\mathcal{U}\phi)(x)$ and \eqref{e1:5}.  

The upshot of the above consideration is that if $d\mu=\phi\,d\beta$ and $d\nu=\psi\,d\beta$ are two probability measures with $\phi,\psi$ being $C^{2}$ on $[-2,2]$, then 
\[
\int \left( H\mu-H\nu \right)^{2}d\alpha=4\int (\mathcal{U}(\phi-\psi))^{2}d\alpha=2\int (\phi-\psi)^{2}d\beta.  
\]
It is this formula which inspires the following definition
\begin{equation}\label{e3:6}
\mathcal{I}(\mu,\nu)=\begin{cases} 
\int \left( \frac{d\mu}{d\beta}-\frac{d\nu}{d\beta} \right)^{2}d\beta, & \text{ if }\quad \frac{d\mu}{d\beta},\frac{d\nu}{d\beta}\in L^{2}(\beta) \\
+\infty &\text{otherwise}.  
\end{cases}
\end{equation} 

For a given $L>0$, we define on measures supported on $[-2L,2L]$,  
\begin{equation}\label{e3:6}
\mathcal{I}_{L}(\mu,\nu)=\begin{cases} 
\int \left( \frac{d\mu}{d\beta_{L}}-\frac{d\nu}{d\beta_{L}} \right)^{2}d\beta_{L}, & \text{ if }\quad \frac{d\mu}{d\beta_{L}},\frac{d\nu}{d\beta_{L}}\in L^{2}(\beta_{L}) \\
+\infty &\text{otherwise}.  
\end{cases}
\end{equation}

\begin{theorem}\label{t:3}
\begin{enumerate}
\item For a measure $\mu$ supported on $[-2,2]$, $\frac{d\mu}{d\beta}\in L^{2}(\beta)$ if and only if $H\mu\in L^{2}(\alpha)$.  
\item For any probability measures, $\mu,\nu$ on $[-2,2]$, 
\begin{equation}\label{e3:12}
\mathcal{J}(\mu,\nu)=2\mathcal{I}(\mu,\nu).
\end{equation}

\item For any probability measures $\mu,\nu$ on $[-2,2]$, 
\begin{equation}\label{e3:5}
2\mathcal{H}(\mu,\nu)\le \mathcal{J}(\mu,\nu) = 2\mathcal{I}(\mu,\nu),
\end{equation}
with equality if $\mu(dx)-\nu(dx)=Cx \beta(dx)$. 

\item In particular, by simple rescaling, for any probability measures $\mu,\nu$ on $[-L,L]$, 
\[
2\mathcal{H}(\mu,\nu)\le \mathcal{J}(\mu,\nu)= 2 L^{2}\mathcal{I}_{L}(\mu,\nu). 
\]
\end{enumerate}
\end{theorem}

\begin{proof}
\begin{enumerate}

\item To prove the assertion, we recall two facts from the theory of Hilbert transform.   The definition of the Hilbert transform is given by (notice that this is twice the one appearing in the literature) 
\[
H\mu(x)=p.v.\int \frac{2}{x-y}\mu(dx)=\lim_{\epsilon\to 0}\int\frac{2(x-y)}{(x-y)^{2}+\epsilon^{2}}\mu(dy).
\]
It is a standard result in the theory of Hilbert transform (see for instance \cite{Boo} and reproved in \cite{Loo}) that this is well defined $\lambda$-a.s, for all points $x$, with $\lambda$ being the Lebesgue measure on $\R$.  

The first result we will use states that there is a constant $C>0$ such that for any finite positive measure $\mu$ 
\begin{equation}\label{e3:17}
t\lambda(\{x: |H\mu(x)|\ge t\})\le C\|\mu\|_{v} \text{ for all } t>0. 
\end{equation}
 Here $\| \mu\|_{v}$ is the variation of $\mu$ (i.e. in this case is just $\mu(\R)$). It is a straightforward fact that this can be extended to any finite measure $\mu$ not necessarily positive, with the change that $\|\mu\|_{v}=\mu^{+}(\R)+\mu^{-}(\R)$ where $\mu=\mu^{+}-\mu^{-}$ is the standard decomposition of $\mu$ into the non-negative and non-positive parts.     This fact can be found for instance in \cite{Loo} but it is eventually attributed to Kolmogorov.  
 
The second result we will use is that if $\mu$ is a finite measure and we take $d\mu=fd\lambda+\mu_{s}$, where $\mu_{s}$ is the singular part of the measure $\mu$, then 
 \begin{equation}\label{e3:18}
 \lim_{t\to\infty}\frac{\pi t}{2}\mathbbm{1}_{\{x:|H\mu(x)|\ge t \})}d\lambda=d\mu_{s}
 \end{equation}
 where the convergence is in the sense of weak convergence.   This can be found for instance in \cite[Eq. 5.4]{PS} and \cite[Theorem 1]{Pol}.

Now, we proceed to the proof of the main statement.  Assume first that $\phi=\frac{d\mu}{d\beta}\in L^{2}(\beta)$.  Then we want to prove that $H\mu\in L^{2}(\alpha)$.   To do this, observe that since $\phi\in L^{2}(\beta)$, we can use \eqref{e3:10} to show that there is an extension of $H\mu$ to $L^{2}(\alpha)$.  More precisely, approximate $\phi$ with a sequence of smooth positive functions $\zeta_{n}$ in $[-2,2]$ with $\int \zeta_{n}d\beta=1$ and this in turn, using \eqref{e3:10}, shows that $H\zeta_{n}$ converges in $L^{2}(\alpha)$.  So what it remains to show is that $H\mu=\lim_{n\to\infty}H\zeta_{n}$.   For this last part, observe that because $\zeta_{n}$ converges to $\phi$ in $L^{2}(\beta)$, it is relatively easy to show that 
\[
\int\frac{|\zeta_{n}(x)-\phi(x)|}{\pi\sqrt{4-x^{2}}}\lambda(dx)\xrightarrow[n\to\infty]{}0.
\]
Combining this with \eqref{e3:17}, we obtain that (here $\nu_{n}=\zeta_{n}d\beta$)
\[
\lim_{n\to\infty}\lambda(\{ x:|H\mu(x)-H\nu_{n}(x)|\ge t \})=0
\]
which yields the convergence in measure of $H\nu_{n}$ toward $H\mu$ as $n\to\infty$.  Since, $H\nu_{n}$ converges also in $L^{2}(\alpha)$, it means that in fact $H\nu_{n}$ converges in $L^{2}(\alpha)$ and $H\mu\in L^{2}(\alpha)$ which is what we wanted.

For the reverse implication, assume now we have $H\mu\in L^{2}(\alpha)$.   The first step is to show that $\mu$ is absolutely continuous with respect to $\beta$ or alternatively with respect to $\lambda$.   To this end, we will exploit \eqref{e3:18}.   Indeed for any continuous function $f$ supported on $[-2,2]$, 
\[
\lim_{t\to\infty}\frac{\pi t}{2}\int \mathbbm{1}_{\{x:|H\mu(x)|\ge t \})}f(x)dx=\int fd\mu_{s}.
\]
On the other hand, since $f$ is supported on $[-2,2]$, we can apply Holder's inequality followed by Chebyshev's to obtain 
\[
\int \mathbbm{1}_{\{x:|H\mu(x)|\ge t \})}f(x)dx\le \lambda(\{x\in[-2,2]:|H\mu(x)|\ge t \})^{1/p}\| f\|_{q} \le \frac{\| f\|_{q}}{t^{r/p}}\left( \int_{-2}^{2} |H\mu|^{r}d\lambda \right)^{1/p}
\]
for any $p,q\ge1$ with $1/p+1/q=1$ and $r>p$.  In fact, we choose $p,q,r$ such that $1<p<r<4/3$ and using again Holder's inequality,  we continue writing 
\[
\int_{-2}^{2} |H\mu|^{r}d\lambda\le \left(\int_{-2}^{2}|H\mu(x)|^{2}\alpha(dx)\right)^{r/2}\left( \int _{-2}^{2}\left(\frac{4\pi^{2}}{4-x^{2}}\right)^{\frac{r}{2(2-r)}} \right)^{(2-r)/2} \le C_{r} \| H\mu\|_{L^{2}(\alpha)}^{r}.
\]
The conclusion we draw from these estimates is that 
\[
\frac{\pi t}{2}\left|\int \mathbbm{1}_{\{x:|H\mu(x)|\ge t \})}f(x)dx\right|\le \frac{C}{t^{r/p-1}}\| H\mu\|_{L^{2}(\alpha)}^{r/p}\| f\|_{q}
\]
and thus, letting $t\to\infty$, 
\[
\int f d\mu_{s}=0,
\]
which is nothing but the fact that $\mu$ is absolutely continuous with respect to $\lambda$ and thus with respect to  $\beta$.  

To go forward,  we need to show that $\phi=\frac{d\mu}{d\beta}$ is in $L^{2}(\beta)$, or alternatively that $\phi-1\in L^{2}(\beta)$.  To achieve this, consider $L^{2}_{0}(\beta)$ as the set of $L^{2}$ functions of mean zero with respect to $\beta$ and let $L^{2}_{0}(\alpha)$ be  the set of functions of $L^{2}$ functions with mean zero with respect to $\alpha$.   We now show that $H_{\beta}$ extends from the smooth functions in $L^{2}_{0}(\beta)$ into $L^{2}_{0}(\alpha)$ such that 
\begin{equation}\label{e3:20}
\|H_{\beta}\psi\|_{L^{2}(\alpha)}=\frac{1}{\sqrt{2}}\|\psi\|_{L^{2}(\beta)}
\end{equation}
for all $\psi\in L^{2}_{0}(\beta)$.   From \eqref{e3:10}, it is clear that there is an operator $L:L^{2}_{0}(\beta)\to L^{2}_{0}(\alpha)$ such that $L$ coincides with $H_{\beta}$ on smooth functions.  The point is to show that $L\phi=H_{\beta}\phi$ for any function $\phi\in L^{2}_{0}(\beta)$.   This can be done as follows.  Take a $\phi\in L^{2}_{0}(\beta)$ and approximate it in $L^{2}(\beta)$ with some smooth functions $\xi_{n}\in L^{2}_{0}(\beta)$.  Then, from the equality \eqref{e3:10} on smooth functions, it follows that $H_{\beta} \xi_{n}$ is a Cauchy sequence in $L^{2}_{0}(\alpha)$.  On the other hand, using that $\xi_{n}$ converges in $L^{1}(\beta)$ combined with \eqref{e3:17}, we conclude that $H_{\beta}\xi_{n}$ forms a Cauchy sequence in measure and thus, its limit in measure must be the same as its limit in $L^{2}(\alpha)$, from which we deduce that $L\phi=H_{\beta}\phi$.  

Once the isometry property above is established, the fact that $\phi\in L^{2}(\beta)$ follows now easily.

\item We just proved this above as in equation \eqref{e3:20}.  

\item It is enough to consider the case $\mathcal{I}(\mu,\nu)$ finite, in which case we certainly have that both $\mu,\nu$ have densities $\frac{d\mu}{d\beta},\frac{d\nu}{d\beta}\in L^{2}(\beta)$.  Writing,  $\mu=\phi d\beta$ and $\nu=\psi d\beta$ with $\phi,\psi\in L^{2}(\beta)$, the inequality to be proved becomes equivalent to 
\[
\langle \mathcal{E}(\phi-\psi),(\phi-\psi) \rangle\le \langle \phi-\psi,\phi-\psi \rangle
\]
which follows from the fact that the spectrum of $\mathcal{E}$ (which is a bounded selfadjoint operator on $L^{2}(\beta)$) restricted to $L^{2}_{0}(\beta)$ is $\{1/n;n\ge1 \}$.   Clearly, the equality is attained if $\phi(x)-\psi(x)=Cx$, the same thing as $\mu(dx)-\nu(dx)=Cx \beta(dx)$.

\item Follows by scaling. \qedhere  
\end{enumerate}
\end{proof}

Another consistent argument for the choice of the Fisher information is given by the following analogy with the classical case.  

In the classical case of the Ornstein-Uhlenbeck operator $L$,  the connection between the entropy and the Fisher information is given by the fact that the Fisher information appears naturally as the derivative of the entropy along the semigroup generated by $L$.  

We want to draw a similar picture in the case of measures on the interval $[-2,2]$ with the role of the Ornstein-Uhlenbeck operator $L$ being played by the counting number operator $\mathcal{N}$  on functions on $[-2,2]$.  

What we have in mind here is the following.   The semigroup generated by the counting number operator $\mathcal{N}$ is $\mathcal{P}_{t}$ and can be shown to be computed as 
\[
\mathcal{P}_{t}f(x)=\int k_{t}(x,y)f(y)\beta(dy)
\]
with 
\[
k_{t}(x,y)=1+2\sum_{n\ge1}e^{-tn}\phi_{n}(x)\phi_{n}(y)
\]
where the factor $2$ in front of the summation comes from the fact that $\int \phi_{n}^{2}\,\beta(dy)=1/2$.   

In the classical case, the derivative of the entropy along the semigroup $\mathcal{P}_{t}$ is exactly the Fisher information.  The same phenomena holds true in this local versions of the entropy and Fisher information.   

More precisely, assume that $\mathcal{I}(\mu,\nu)$ is finite.  This means that both measures $\mu,\nu$ have densities in $L^{2}(\beta)$ and taking $\phi=\frac{d\mu}{d\beta}-\frac{d\nu}{d\beta}$ we can write 
\[
\mathcal{H}(\mu,\nu)=\langle \mathcal{E}\phi,\phi \rangle.  
\]
If we take the measures $\mu_{t}=(\mathcal{P}_{t}\phi)\,d\beta$ and similarly $\nu_{t}=(\mathcal{P}_{t}\phi) \, d\beta$, then 
\[
\frac{d}{dt}\mathcal{H}(\mu_{t},\nu_{t})=\frac{d}{dt}\langle \mathcal{E}\mathcal{P}_{t}\phi,\mathcal{P}_{t}\phi \rangle=\langle \mathcal{E}\mathcal{N}\mathcal{P}_{t}\phi,\mathcal{P}_{t}\phi \rangle=\langle \mathcal{P}_{t}\phi,\mathcal{P}_{t}\phi \rangle =\mathcal{I}(\mu_{t},\nu_{t}).
\]

It is the content of Theorem~\ref{t:3} which actually unifies the two versions of the Fisher information, the one in terms of the Hilbert transform and the other one in terms of the densities of the measures involved.   

Thus the two points of view outlined above converge toward the same thing and gives a consistent notion of Fisher information in this framework.

\section{$L^{p}$ considerations around local Log-Sobolev and an open problem}\label{s:6}

From the discussion in the previous section we learn that the local Log-Sobolev hold in the form 
\[
2\mathcal{H}(\mu,\nu)\le \int|H\mu-H\nu|^{2}d\alpha.
\]
It is natural to ask if this still remains true if we replace the $L^{2}$ norm on the right hand side by another $L^{p}$ norm. In other words is it true that there is a $1\le p <2$ and a constant $C_{p}>0$ such that for all probability measures $\mu,\nu$ on $[-2,2]$, 
\begin{equation}\label{e6:1}
C_{p}\mathcal{H}(\mu,\nu)\le \left(\int|H\mu-H\nu|^{p}d\alpha\right)^{2/p}.
\end{equation}

We do not know the answer to this question, but want to show that for $p<3/2$ this is not possible.  This is based on the example given by measures $\mu,\nu$ such that for $0<r<1$ and a small constant $\eta$
\[
\mu(dx)-\nu(dx)=\eta\left(\sum_{n\ge1}r^{n-1}T_{n}(x/2)\right)\beta(dx).
\]
From \eqref{e3:11}, we get that 
\[
H(\mu-\nu)(x)=\eta\left(\sum_{n\ge1}r^{n-1}U_{n-1}(x/2)\right)=\frac{\eta}{1-rx+r^{2}}.
\]
With this choice and Proposition~\ref{p:1}, we obtain that for $0<r<1$, 
\[
\mathcal{H}(\mu,\nu)=2\eta^{2}\sum_{n\ge1}\frac{r^{2(n-1)}}{n}=-2\eta^{2}\log(1-r^{2}).
\]
Now for $p<3/2$,
\[
\begin{split}
\int|H\mu-H\nu|^{p}d\alpha&=\frac{\eta^{p}}{2\pi}\int_{-2}^{2}\frac{1}{(1-rx+r^{2})^{p}}\sqrt{4-x^{2}}dx=16\eta^{2}\int_{0}^{1}\frac{\sqrt{u(1-u)}}{(4ru+(1-r)^{2})^{p}}du \\ 
&\le \frac{16\eta^{p}}{(4r)^{p}}\int_{0}^{1}u^{1/2-p}du<\infty
\end{split}
\]
where in the middle we made the change of variable $x=2-4u$, $0\le u \le 1$.  The moral of this calculation is that \eqref{e6:1} can not hold true with $1\le p<3/2$ because the right hand side is bounded in $r\in(0,1)$ and the left hand side blows up as $r$ approaches $1$ from below.   

For $r=3/2$, using Mathematica, we obtain that 
\[
\begin{split}
\int|H\mu-H\nu|^{3/2}d\alpha&=16\eta^{3/2}\int_{0}^{1}\frac{\sqrt{u(1-u)}}{(4ru+(1-r)^{2})^{3/2}}du\\ 
&=\frac{\eta^{3/2}}{8} (-4+6\log(2)-2 \log(1-r))+O((\log(1-r))^{2})
\end{split}
\]
which does not rule out \eqref{e6:1}.  

For $p=3/2$ we do not have a counterexample to \eqref{e6:1} nor a proof of validity of it.   We post this as an open problem here.  

\begin{problem} There is a constant $C>0$ such that for any probability measures $\mu,\nu$ on $[-2,2]$, 
\[
C\mathcal{H}(\mu,\nu)\le \left(\int|H\mu-H\nu|^{3/2}d\alpha\right)^{4/3}.
\]
\end{problem}

A positive answer to this question would give the optimal $p$ for which \eqref{e6:1} is true.    A negative answer would continue with the following. 

\begin{problem} Is there a $3/2< p<2$ such that for some constant $C_{p}>0$, \eqref{e6:1} holds true for any probability measures $\mu,\nu$?   
And if so, what is the smallest such $p$?  
\end{problem}

A reformulation of these open problems in terms of trigonometric series can be done based on \eqref{e3:11}, \eqref{e3:100} and the definition of the Chebyshev polynomials of second kind.   Using an approximation of the measures $\mu,\nu$ with measures of the form $\phi\, d\beta$ and $\psi\, d\beta$, an equivalent form  of \eqref{e6:1} is the following.  What is the smallest $1<p$ such that for any $n\ge1$ and $a_{1},a_{2},\dots, a_{n}\in\R$, 
\begin{equation}
C_{p}\sum_{k=1}^{n}\frac{a_{k}^{2}}{k}\le \left(  \int_{0}^{\pi}\left|\sum_{k=1}^{n}a_{k}\sin(k t)\right|^{p}\sin^{2-p}(t)dt \right)^{2/p}?  
\end{equation}
The conclusion of this section is that definitely $3/2\le p$, but it is not clear that the smallest $p$ is exactly $3/2$.

\section{HWI inequality}\label{s:5}

This section is dedicated to a version of the celebrated HWI from \cite{OV}.  The statement is the following.  

\begin{theorem}\label{t:4}
For any probability measures $\mu,\nu$ on $[-2,2]$, 
\begin{equation}\label{e4:1}
\mathcal{H}(\mu,\nu)\le \sqrt{2\mathcal{I}(\mu,\nu)}W_{1}(\mu,\nu)-\frac{1}{2}W_{1}^{2}(\mu,\nu)
\end{equation}
with equality if $\mu(dx)-\nu(dx)=Cx \beta(dx)$.
\end{theorem}
 
 \begin{proof}
 It is sufficient to prove this in the case $\mathcal{I}(\mu,\nu)$ is finite.  Thus, let $\eta=\frac{d\mu}{d\beta}$ and $\zeta=\frac{d\nu}{d\beta}$ with $\eta,\zeta\in L^{2}(\beta)$.   Let $\psi=\eta-\zeta$.  Clearly, $\psi\in L^{2}(\beta)$.   
 
 Here, the main observation is that conform Corollary~\ref{c:2}, the inequality we need to prove writes equivalently as 
 \[
 \langle \mathcal{E}\psi,\psi \rangle \le 2\sqrt{\langle \psi,\psi \rangle \langle\mathcal{E}^{2}\psi,\psi \rangle} - \langle\mathcal{E}^{2}\psi,\psi \rangle.  
 \]
 
Writing $\psi=\sum_{n\ge1}\alpha_{n}\phi_{n}$, this can be reinterpreted as 
\[\tag{*}
\left(\sum_{n\ge1}\alpha_{n}^{2}(1/n+1/n^{2})\right)^{2}\le 4\sum_{n\ge1}\alpha_{n}^{2}\sum_{n\ge1}\alpha_{n}^{2}/n^{2}.
\]
To justify this, apply Cauchy-Schwartz as 
\[
\left(\sum_{n\ge1}\alpha_{n}^{2}(1/n+1/n^{2})\right)^{2}\le \sum_{n\ge1}\alpha_{n}^{2}\sum_{n\ge 1}\alpha_{n}^{2}(1/n+1/n^{2})^{2}
\]
and then (*) follows from the fact that $(1/n+1/n^{2})^{2}\le 4/n^{2}$.  Tracing back all inequalities, we see that equality follows for the case $\mu(dx)-\nu(dx)=Cx\beta(dx)$.  \qedhere
\end{proof}
 
 \section{Remarks}  
 
 \begin{enumerate}
\item  It is interesting to point out that these local versions of the functional inequalities are essentially on intervals.  Taking an arbitrary set, say $K$, for instance a finite union of intervals, the local transportation still holds for all measures supported on $K$.  This can be easily seen by simply considering the set $K$ as a subset of an interval.  Perhaps the interesting thing to follow on here is the significance of the best constant in the inequality and the measures for which this is achieved.  
 
\item It is interesting to figure out a similar global version of the Log-Sobolev inequality.  This should have some connection with the global transportation inequality. 

\item For the local Log-Sobolev, if we take this on an arbitrary set $K$, then it would be nice to see a similar picture as in the case of the interval.   It is not clear what the natural replacement of the measure $\alpha$
 from the definition of $\mathcal{J}$ should be.  Even in the interval case, it is somehow an interesting play between the semicircle and the arcsine laws, whose replacement is not obvious for an arbitrary set $K$.    
 
 \end{enumerate}

 \vspace{1cm}
{\bf Thanks. }{ I want to thank Florent Benaych-Georges for pointing to the author the work of \cite{MS}.  I also want  to thank  \' Edouard Maurel-Segala and Myl\`ene Ma\"ida for their comments and particularly for the question from Remark~\ref{r:10}.}

\end{document}